\documentclass[11pt]{article}
\usepackage[colorlinks=false]{hyperref}
\usepackage{amsfonts, amsmath, amssymb, amsthm, tikz, caption, subcaption}
\usepackage[shortlabels]{enumitem}
\usepackage[all]{xy}

\usetikzlibrary{plotmarks, calc, arrows.meta, decorations.pathmorphing}

\definecolor{darkgray}{RGB}{64,64,64}
\definecolor{litegray}{RGB}{192,192,192}
\definecolor{red}{rgb}{.859375,.265625,.21484375}
\definecolor{blue}{rgb}{.26171875,.51953125,.95703125}

\title{Finding a best approximation pair of points for two polyhedra}
\author{Ron Aharoni\thanks{Department of Mathematics, the Technion -- Israel Institute of Technology, Technion City, Haifa 3200003, Israel.}${\ }^{,}$\footnote{Supported in part by the United States--Israel Binational Science Foundation (BSF) grant no. 2012031, the Israel Science Foundation (ISF) grant no. 2023464 and the Discount Bank Chair at the Technion. Email: {\tt ra@tx.technion.ac.il}.} \and Yair Censor\footnote{Department of Mathematics, University of Haifa, Mt. Carmel, Haifa 3498838, Israel. Supported in part by BSF grants no. 2013003. Email: {\tt yair@math.haifa.ac.il}.} \and Zilin Jiang\footnotemark[1]${\ }^{,}$\footnote{Supported in part by ISF grant nos 1162/15, 936/16. Email: {\tt jiangzilin@technion.ac.il}.}}
\date{}

\newtheorem{theorem}{Theorem}

\newtheorem{lemma}[theorem]{Lemma}

\theoremstyle{definition}

\theoremstyle{remark}
\newtheorem{remark}{Remark}

\newcommand{\N}{\mathbb{N}}
\newcommand{\rd}{\mathbb{R}^d}

\newcommand{\abs}[1]{\left\lVert {#1}\right\rVert}

\newcommand{\dset}[2]{\left\{#1\mid#2\right\}}
\newcommand{\sset}[1]{\left\{#1\right\}}
\newcommand{\dseq}[4]{\sset{#1}_{#2=#3}^{#4}}
\newcommand{\sseq}[4]{\dseq{{#1}_{#2}}{#2}{#3}{#4}}

\begin{document}

\maketitle

\begin{abstract}
  Given two disjoint convex polyhedra, we look for a best approximation pair relative to them, i.e., a pair of points, one in each polyhedron, attaining the minimum distance between the sets. Cheney and Goldstein showed that alternating projections onto the two sets, starting from an arbitrary point, generate a sequence whose two interlaced subsequences converge to a best approximation pair. We propose a process based on projections onto the half-spaces defining the two polyhedra, which are more negotiable than projections on the polyhedra themselves. A central component in the proposed process is the Halpern--Lions--Wittmann--Bauschke algorithm for approaching the projection of a given point onto a convex set.
\end{abstract}



\section{Introduction}\label{sec:Intro}

A \textit{best approximation pair }relative to two closed convex sets $A$ and $B$ is a pair $(a,b)\in A\times B$ attaining $\abs{a-b}=\min\abs{A-B}$, where $A-B:=\dset{x-y}{x\in A, y\in B}$.

For a closed convex set $C$ denote by $P_{C}$ the metric projection operator onto $C$. Take an arbitrary starting point $a_{0}\in\rd$, the $d$-dimensional Euclidean space, and consider the sequence:
\[
\xymatrix@R=0pt{a_{0}\ar[rdd]_{P_{B}} &  & a_{2}\ar[rdd]_{P_{B}} &  &  &  & a_{2k}\ar[rdd]_{P_{B}} &  & a_{2k+2}\\
 &  &  &  & \cdots &  &  &  &  & \cdots\\
 & b_{1}\ar[uur]_{P_{A}} &  & b_{3} &  & b_{2k-1}\ar[uur]_{P_{A}} &  & b_{2k+1}\ar[uur]_{P_{A}}
}
\]

A well-known theorem of Cheney and Goldstein~\cite{MR0105008} specifies conditions under which alternating metric projections onto the two sets are guaranteed to converge to the best approximation pair. In fact, their result applies when $A$ and $B$ are closed convex sets in a Hilbert space, and one of them is compact. For related results, using the averaged alternating reflections method and applying it to not necessarily convex sets, see \cite{MR2058156, MR2425037}. For a study of von Neumann's alternating projection algorithm for two sets, see \cite{MR1239403} and \cite{MR3034852}. For best approximation in general, we refer the reader to Deutsch's excellent book~\cite{MR1823556}. A recent review of algorithms for inconsistent feasibility problems, some of which relevant to the work presented here, appears in \cite{Censor:2018aa}.

In real-life problems, convex polyhedra are usually represented by a set of linear constraints, namely as the intersection of half-spaces. Projecting onto the polyhedron can then be done using projections onto the half-spaces. We propose a projection method for finding the best approximation pair that uses directly projections onto the half-spaces, instead of on the polyhedra.

An algorithm for approaching the projection of a point $a$ onto a polyhedron, using projections onto the half-spaces defining the polyhedron, was proposed by Halpern, Lions, Wittmann and Bauschke (HLWB)\footnote{This acronym was dubbed in \cite{MR2232923}.} \cite{MR0218938, MR0470770, MR1156581, MR1402593}. The HLWB algorithm works by projecting successively and cyclically onto the half-spaces, the main stratagem being that after each projection the algorithm ``pulls'' a bit back in the direction of $a$. The latter guarantees that the algorithm does not ``forget'' the point $a$ whose projection onto the polyhedron is sought after.

We propose and study the convergence of an iterative process based on projections onto the individual half-spaces defining the polyhedra, which are more negotiable than projections on the polyhedra themselves. We apply the HLWB algorithm alternatingly to the two polyhedra. Its application is divided into sweeps --- in the odd numbered sweeps we project successively onto half-spaces defining $A$, and in even numbered sweeps onto half-spaces defining $B$. A critical point is that the number of successive projections onto each set's half-spaces increases from sweep to sweep. The proof of convergence of the algorithm is rather standard in the case that the best approximation pair is unique. The non-uniqueness case, however, poses some difficulties and its proof is more involved.

The algorithm belongs to a family known as \textit{projection methods}. These are iterative algorithms that use projections onto individual sets, to converge to a point in the intersection of these sets, or images of them under some transformation. They were originally used to solve systems of linear equations in Euclidean space (see, e.g., \cite{MR2364086, MR2931682, censor1997parallel, MR2849884, MR2041220}), and later were  extended to solve general convex feasibility problems in a Hilbert space, see, e.g., \cite{MR3616647}. On the low computational cost of projection methods, see \cite{MR3155358, MR2891928}. Consult also \cite{MR3391223}.

The paper is organized as follows. In Section~\ref{algorithm} we present the alternating HLWB algorithm. In Section~\ref{proof} we prove some preliminary results needed for the proof of convergence of the algorithm. The proof itself is given in Section~\ref{sec:Proof}. Finally, in Section~\ref{discussion} we discuss possible choices of the parameters for the algorithm.

\section{An Alternating HLWB Algorithm}\label{algorithm}

Throughout the paper, we assume that $A:=\cap_{i=1}^{M}A_{i}$ and $B:=\cap_{j=1}^{N}B_{j}$ are two nonempty convex polyhedra, where $\sseq{A}{i}{1}{M}$ and $\sseq{B}{j}{1}{N}$ are two families of closed half-spaces. By adding $A_{i}$ or $B_{j}$ that are equal to the entire space $\rd$ (or alternatively repeat the same half-space) we may assume that $M=N$. For the purpose of performing unboundedly many projections, we extend the sequences $\sset{A_{i}} $ and $\sset{B_{j}} $ to all $i, j\in\N$ by the rules $A_{i}=A_{i\bmod N}$ and $B_{j}=B_{j\bmod N}$, where the $\bmod N$ function takes values in $\sset{1,2,\dots,N}$.

We incorporate into our algorithm the HLWB algorithm, which is designed to find the projection of a point $a$ onto a polyhedron $C$, using the projections onto the half-spaces defining $C$. Let $P_1, P_2, \ldots ,P_N$ be the respective projections onto these half-spaces. The HLWB algorithm starts by choosing an arbitrary starting point $x_0$ and numbers $\lambda_n$ satisfying: 
\begin{equation}\label{ctrl-1}
\lim_{n\to\infty}\lambda_{n}=0,\quad{\textstyle \sum_{n}\lambda_{n}=\infty,\quad{\textstyle \sum_{n}|{\lambda_{n}-\lambda_{n+N}}|<\infty.}}
\end{equation}
A sequence $\sseq{x}{n}{1}{\infty}$ is then recursively generated by the rule:
\begin{equation*}
x_{n}:=\lambda_{n}a+(1-\lambda_{n}){P_{n\bmod N}}(x_{n-1}),
\end{equation*}

Bauschke~\cite[Theorem 3.1]{MR1402593} proved that the sequence
$\sseq{x}{n}{0}{\infty}$ generated by this HLWB algorithm convergences to $P_{C}(a)$. Some computational performance results with the HLWB and the Dykstra~\cite{MR727568} algorithms were presented in \cite{MR2232923}.

In our proposed algorithm, we apply the HLWB algorithm alternatingly to $A$ and to $B$. We call this method ``A-HLWB'' (``A'' for ``alternating''). Like in HLWB, we choose numbers $\lambda_{n}$ satisfying \eqref{ctrl-1}. For points $a,x\in\rd$ and $n\in\N$, we recursively define
\[
  Q_{B,0}(a;x):=x\quad\text{and}\quad 
  Q_{B,n}(a;x):=\lambda_{n}a+(1-\lambda_{n})P_{B_{n}}(Q_{B,n-1}(a;x)).
\]
For $b,x\in\rd$ and $n\in\N$, $Q_{A,n}(b;x)$ is similarly defined
\[
  Q_{A,0}(b;x):=x\quad\text{and}\quad 
  Q_{A,n}(b;x):=\lambda_{n}b+(1-\lambda_{n})P_{A_{n}}(Q_{A,n-1}(b;x)).
\]
Thus, $Q_{A,n}(b;x)$ and $Q_{B,n}(a;x)$ are operators, each being defined by a
sequence of $n$ iterations.

We also choose an arbitrary starting point $a_{0}\in\rd$ and a non-decreasing
sequence $(n_{k})$ such that
\begin{equation}\label{iter}
  n_{k} \to \infty \quad \text{and} \quad
  \sup_{k_{0}}\sset{\sum_{k>k_{0}}\prod_{n>n_{k_{0}}}^{n_{k}}(1-\lambda_{n})} <
  \infty.
\end{equation}
Once the sequence $(\lambda_{n})$ is chosen so that \eqref{ctrl-1} holds, one can always make $(n_{k})$ increase rapidly enough so that \eqref{iter} holds. For example, $\lambda_{n}=\tfrac{1}{n+1},n_{k}=\lfloor1.1^{k}\rfloor$ satisfy both \eqref{ctrl-1} and \eqref{iter}. To see why these parameters satisfy the second inequality in \eqref{iter}, simply notice that $$\sum_{k > k_0}\prod_{n > n_{k_0}}^{n_k}(1-\lambda_n) = \sum_{k > k_0}\prod_{n > n_{k_0}}^{n_k}\frac{n}{n+1} = \sum_{k > k_0}\frac{n_{k_0}+1}{n_k+1} \approx \sum_{k > k_0}\frac{1.1^{k_0}}{1.1^k} = 10.$$

The $k$th sweep of the A-HLWB algorithm uses $n_k$ iterations of the HLWB algorithm to generate:
\begin{equation}\label{round}
  b_{k+1}=Q_{B,n_{k}}(a_{k};a_{k}')\text{ if }k\text{ is even}; \quad
  a_{k+1}=Q_{A,n_{k}}(b_{k};b_{k}')\text{ if }k\text{ is odd},
\end{equation}
where the auxiliary parameter $a_{k}'$ or $b_{k}'$ is chosen before each sweep. The validity of the algorithm is guaranteed if the auxiliary sequence $(a_{2k}',b_{2k+1}')$ is bounded. For example, one may simply take $a_{2k}'=b_{2k+1}'=a_{0}$. We now state our main convergence result.

\begin{theorem}\label{main}
  If the above assumptions on the mappings and on the parameters hold and the auxiliary sequence $(a_{2k}',b_{2k+1}')$ is bounded, then the pairs $(a_{2k},b_{2k+1})$, generated by the A-HLWB algorithm~\eqref{round}, converge to a best approximation pair relative to $(A,B)$.
\end{theorem}

The second inequality in \eqref{iter} is technical and could possibly be redundant. In fact, if the best approximation pair is unique, the convergence is assured without this inequality (see Remark~\ref{cor}). However, we are unable to remove it for Theorem~\ref{main} in the non-uniqueness case.

\section{Preliminaries for the Proof of Convergence}\label{proof}

We present several preliminary results that will be used to prove Theorem~\ref{main}, the first of which says, in Lemma \ref{compact} below, that the set of points generated by the A-HLWB algorithm is bounded. This follows from a result of Aharoni, Duchet and Wajnryb~\cite{MR757629}, see also Meshulam~\cite{MR1395472}.

\begin{theorem}[Theorem of Aharoni, Duchet and Wajnryb~\cite{MR757629}]\label{adw}
  Any sequence of points in $\rd$ obtained by successive projections of a point onto elements of a finite set of hyperplanes is bounded.
\end{theorem}

\begin{lemma}\label{compact}
  For every bounded set $D\subset\rd$, there exists a compact set $C\subset\rd$ containing $D$ such that $Q_{A,m}(b;x),Q_{B,n}(a;x)\in C$ for all $a,b,x\in C$ and $m,n\in\N$.
\end{lemma}

\begin{proof}
  Take a simplex with vertices $x_{0},x_{1},\dots,x_{d}\in\rd$ that contains $D$. Denote the bounding hyperplanes of the half-spaces $A_{i}$ and $B_{i}$ by $\partial A_{i}$ and $\partial B_{i}$, respectively. Let $X$ be the set of points obtained by successive projections of $x_{0},x_{1},\dots,x_{d}$ on $\sseq{\partial A}{i}{1}{M}$ and on $\sseq{\partial B}{i}{1}{N}$. By Theorem~\ref{adw}, we know that $X$ is bounded, and so is its convex hull $Y:=\mathrm{conv}(X)$.

  Notice that $Q_{A,m}(b;x)$ is either $\lambda_{m}b+(1-\lambda_{m})Q_{A,m-1}(b;x)$ or $\lambda_{m}b+(1-\lambda_{m})P_{\partial A_{m}}Q_{A,m-1}(b;x)$ depending on whether $Q_{A,m-1}(b;x)$ is in $A_{m}$. One can then show by induction on $m$ that $Q_{A,m}(b;x)\in Y$ for every $b,x\in Y$.

  The same argument shows that $Q_{B,n}(a;x)\in Y$ for every $a,x\in Y$. Finally, let $C$ be the closure of $Y$ in $\rd$. Since $Q_{A,m}$ and $Q_{B,n}$ are continuous, $Q_{A,m}(b;x)$ and $Q_{B,n}(a;x)$ are in $C$ for every $a,b,x\in C$.
\end{proof}

The following result is well-known, see, e.g., \cite[Theorem 3]{MR0105008}.

\begin{theorem}\label{lipschitz}
  If $B$ is a closed convex set in Hilbert space, then the projection map $P_{B}$ onto $B$ satisfies the Lipschitz condition $\abs{P_{B}(x)-P_{B}(y)}\le\abs{x-y}$, equality holding only if $\abs{x-P_{B}(x)}=\abs{y-P_{B}(y)}$.
\end{theorem}

The classical 1959 result of Cheney and Goldstein, repeatedly referred to in this paper, is given next as a paraphrased version of their Theorems 2 and 4.

\begin{theorem}\label{cg}
  Let $A$ and $B$ be two closed convex sets in Hilbert space. A point of $A$ is nearest to $B$ if and only if it is a fixed point of $P_{A}P_{B}$. If one set is finite-dimensional and the distance between the sets is attained, then convergence of $((P_{A}P_{B})^{n}(x))$ to a fixed point of $P_{A}P_{B}$ is assured.
\end{theorem}

We need also the following result, which appeared in \cite[Lemma 2.2]{MR1305442}.

\begin{lemma}\label{fixed}
  Let $A$ and $B$ be two closed convex sets in Hilbert space, one of which being finite-dimensional. Suppose that the distance between the sets is attained. If $S$ is a nonempty compact set such that $P_{A}P_{B}(S)=S$, then $S$ consists of points of $A$ nearest to $B$.
\end{lemma}

\begin{proof}
  Define $S':=\dset{s\in S}{P_{A}P_{B}(s)=s}$. Since $S$ is compact and $P_{A}P_{B}(S)\subset S$, by the second part of Theorem~\ref{cg}, for any $x\in S$, $(P_{A}P_{B})^{n}(x)$ converges in $S$ and its limit is a fixed point of $P_{A}P_{B}$.

  Since $S$ is nonempty, so is $S'$ and it is easy to see that $S'$ is compact as well.

  Let $d:=\max_{s\in S}\inf\abs{s-S'}$ and let $y\in S$ be such that $\inf\abs{y-S'}=d$. Since $P_{A}P_{B}(S)\supset S$, there exists $x\in S$ such that $P_{A}P_{B}(x)=y$. Since $\min\abs{x-S'}\le d$, we can take $s'\in S'$ such that $\abs{x-s'}\le d$.

  By way of contradiction, assume that $x\notin S'$. By the first part of Theorem~\ref{cg},
  \begin{equation*}
    \abs{P_{B}(s')-s'}=\inf\abs{A-B}<\abs{P_{B}(x)-x}.
  \end{equation*}
  By Theorem~\ref{lipschitz}, we obtain
  \begin{equation*}
    \abs{y-s'} = \abs{P_{A}P_{B}(x)-P_{A}P_{B}(s')} \le
    \abs{P_{B}(x)-P_{B}(s')} < \abs{x-s'}.
  \end{equation*}
  This contradicts with $\abs{x-s'}\le d\le\abs{y-s'}$. Therefore, $x\in S'$, hence $y=x$ and $d=0$. By the first part of Theorem~\ref{cg}, $S=S'$ implies that $S$ consists only of points of $A$ nearest to $B$.
\end{proof}

The last ingredient that will be used in our proof of Theorem \ref{main} is the following.

\begin{theorem}\label{thm7}
  Let $B$ be a polyhedron in Hilbert space, and assume that $B=\cap_{i=1}^{N}B_{i}\neq\emptyset$, where $\sseq{B}{i}{1}{N}$ are closed convex sets. If the sequence $(\lambda_{n})$ satisfies \eqref{ctrl-1}, then $\lim_{n\to\infty}\abs{Q_{B,n}(a;x)-P_{B}(a)}=0$ for any points $a$ and $x$.
\end{theorem}

\begin{proof}
  This follows from Bauschke's Theorem 3.1 in \cite{MR1402593}. In fact, Bauschke's theorem applies to a broader setting, in which the $B_{i}$'s are sets of fixed points of nonexpansive mappings in Hilbert space.
\end{proof}

It is easy to check that
\begin{equation*}
  \abs{Q_{B,n}(a;x)-Q_{B,n}(a';x')} \le \abs{a-a'} + \abs{x-x'}
\end{equation*}
for all $n$. Together with the fact that $P_{B}$ is nonexpansive (i.e., $1$-Lipschitz), it is routine to check the uniform convergence of $(Q_{B,n})$ on any compact set, leading to the next lemma.

\begin{lemma}\label{uniform}
  Let $B$ be as in Theorem~\ref{thm7}, and let $C$ be a compact set in the Hilbert space. If the sequence $(\lambda_{n})$ satisfies \eqref{ctrl-1}, then
  \begin{equation*}
    \lim_{n\to\infty}\left(\sup_{a,x\in C}\abs{Q_{B,n}(a;x)-P_{B}(a)}\right)=0.
  \end{equation*}
\end{lemma}

\begin{proof}
  For every $\epsilon > 0$, let $C_0$ be a finite $\epsilon$-covering of the compact set $C$. By Theorem~\ref{thm7}, for every $a_0, x_0 \in C_0$, there is $N(a_0, x_0) \in \N$ such that $\abs{Q_{B,n}(a_0;x_0) - P_B(a_0)} < \epsilon$ for all $n > N(a_0, x_0)$. Set $N := \max_{a_0,x_0\in C_0}N(a_0,x_0)$. Given $a, x \in C$, let $a_0, x_0 \in C_0$ be such that $\abs{a-a_0}, \abs{x-x_0} < \epsilon$. For every $n > N$, since both $Q_{B,n}$ and $P_B$ are nonexpansive, we have
  \begin{multline*}
    \abs{Q_{B,n}(a;x) - P_B(a)} \le \abs{Q_{B,n}(a;x) - Q_{B,n}(a_0;x_0)} + \abs{Q_{B,n}(a_0;x_0) - P_B(a_0)} \\ + \abs{P_B(a_0) - P_B(a)} \le \abs{a-a_0} + \abs{x-x_0} + \epsilon + \abs{a_0 - a} < 4\epsilon. \qedhere
  \end{multline*}
\end{proof}

\section{Convergence of the A-HLWB Algorithm}\label{sec:Proof}

In this section we present a proof of the convergence theorem of the A-HLWB algorithm.

\begin{proof}[Proof of Theorem~\ref{main}]
  In order to prove this theorem we inspect the set of accumulation points of $(a_{2k})$. We show that it is compact, fixed under $P_{A}P_{B}$ and, finally, that it is a singleton. By Lemma~\ref{compact}, there exists a compact set $C\subset\rd$ containing $\{a_{0}\}\cup\dseq{a_{2k}'}{k}{0}{\infty}\cup\dseq{b_{2k+1}'}{k}{0}{\infty}$ such that both $Q_{A,m}$ and $Q_{B,n}$ map $C\times C$ to $C$, hence the sequences $(a_{2k})$ and $(b_{2k+1})$ are contained in $C$.

  Let $S$ be the set of accumulation points of $(a_{2k})$. By the Bolzano--Weierstrass theorem $S\neq\emptyset$. Moreover, since $S$ is closed and $S\subset C$, it is compact.

  We claim that $P_{A}P_{B}(S)=S$. Pick any point $s\in S$ and any $\epsilon>0$. Using Lemma~\ref{uniform} and the first assumption in \eqref{iter} that $n_{k}\to\infty$, one can choose $k$ sufficiently large such that
  \begin{equation*}
    \abs{a_{2k}-s}<\epsilon,\quad\sup_{a,x\in
    C}\abs{Q_{B,n_{2k}}(a;x)-P_{B}(a)}<\epsilon/2,\quad\sup_{a,x\in
    C}\abs{Q_{A,n_{2k+1}}(b;x)-P_{A}(b)}<\epsilon/2.
  \end{equation*}
  In particular, since $b_{2k+1}=Q_{B,n_{2k}}(a_{2k};a_{2k}')$ and $a_{2k+2}=Q_{A,n_{2k+1}}(b_{2k+1};b_{2k+1}')$, we have
  \begin{equation*}
    \abs{P_{B}(a_{2k})-b_{2k+1}}<\epsilon/2,\quad\abs{P_{A}(b_{2k+1})-a_{2k+2}
    }<\epsilon/2.
  \end{equation*}
  By the triangle inequality and the fact that $P_{A}$ and $P_{B}$ are nonexpansive, we obtain that
  \begin{multline}\label{triangle}
      \abs{P_{A}P_{B}(s)-a_{2k+2}} \le \abs{P_{A}P_{B}(s)-P_{A}P_{B}(a_{2k})}
      + \abs{P_{A}P_{B}(a_{2k})-P_{A}(b_{2k+1})} \\ + \abs{P_{A}(b_{2k+1})-a_{2k+2}} 
      \le \abs{s-a_{2k}} + \abs{P_{B}(a_{2k})-b_{2k+1}} +
      \abs{P_{A}(b_{2k+1})-a_{2k+2}} < 2\epsilon.
  \end{multline}
  This implies that $P_{A}P_{B}(s)$ is also an accumulation point of $(a_{2k})$. Thus, $P_{A}P_{B}(S)\subset S$. On the other hand, suppose that $s\in S$ is the limit of the subsequence $(a_{2k_{l}})$. Let $s'\in S$ be an accumulation point of the subsequence $(a_{2k_{l}-2})$. The same argument for $P_{A}P_{B}(S)\subset S$ shows that $P_{A}P_{B}(s')$ is an accumulation point of $(a_{2k_{l}})$, and so is $P_{A}P_{B}(s')=s$. This means that $P_{A}P_{B}(S)\supset S$. By Lemma~\ref{fixed}, $S$ consists of points of $A$ nearest to $B$.

  It remains to be shown that $S$ is a singleton, namely that it contains only one point, which is then the limit of $(a_{2k})$. This is clear if there is only one best approximation pair.

  From here on we consider the case that $A$ and $B$ have parallel closest faces. This situation requiers a deeper and more delicate analysis which we give now.

  \begin{figure}
    \centering
    \begin{minipage}{.7\textwidth}
      \centering
      \begin{tikzpicture}
        \coordinate (A) at (0,0);
        \coordinate (B) at (2,1);
        \coordinate (C) at (1,3);
        \coordinate (T1) at (2/3, 1/3);
        \coordinate (T2) at (4/3, 2/3);
        \coordinate (T3) at (5/3, 5/3);
        \coordinate (T4) at (4/3, 7/3);
        \coordinate (T5) at (2/3, 6/3);
        \coordinate (T6) at (1/3, 3/3);
        \coordinate (LA) at (-2,0);
        \coordinate (LB) at (-2,1);
        \coordinate (LC) at (-1,3);
        \draw[darkgray] (A)--(B)--(C)--(A);
        \fill[litegray, opacity=0.3] (A)--(B)--(C)--(A);
        \node at (-0.5,2) {$A$};
        \draw[darkgray] (T2)--(T3);
        \draw[darkgray] (T4)--(T5);
        \draw[darkgray] (T6)--(T1);
        \fill[litegray, opacity=0.5] (T1)--(T2)--(T3)--(T4)--(T5)--(T6)--cycle;
        \node at (1, 1.3) {$T$};
        \draw[darkgray, latex-] (1,13/6)--node[above] {$v$} (5,13/6);
        \draw[darkgray] (A)--(LA);
        \draw[darkgray, dashed] (B)--(LB);
        \draw[darkgray] (C)--(LC);
        \def\shift{1/3}
        \coordinate (D) at (4, 2-\shift);
        \coordinate (E) at (6, 3-\shift);
        \coordinate (F) at (5, 0-\shift);
        \coordinate (RD) at (7, 2-\shift);
        \coordinate (RE) at (8, 3-\shift);
        \coordinate (RF) at (7, 0-\shift);
        \draw[darkgray] (E)--(D)--(F);
        \draw[darkgray, dashed] (E)--(F);
        \fill[litegray, opacity=0.3] (D)--(E)--(F)--cycle;
        \draw[darkgray] (D)--(RD);
        \draw[darkgray] (E)--(RE);
        \draw[darkgray] (F)--(RF);
        \node at (6,1-\shift) {$B$};
      \end{tikzpicture}
      \captionof{figure}{}
      \label{fig_t}
    \end{minipage}%
    \begin{minipage}{.3\textwidth}
      \centering
      \begin{tikzpicture}
        \coordinate (A) at (0:1);
        \coordinate (B) at (60:1);
        \coordinate (C) at (120:1);
        \coordinate (D) at (180:1);
        \coordinate (E) at (240:1);
        \coordinate (F) at (300:1);
        \def\gap{0.3}
        \coordinate (e1) at (30:\gap);
        \coordinate (e2) at (90:\gap);
        \coordinate (e3) at (150:\gap);
        \coordinate (e4) at (210:\gap);
        \coordinate (e5) at (270:\gap);
        \coordinate (e6) at (330:\gap);
        \fill[litegray, opacity=0.5] (A)--(B)--(C)--(D)--(E)--(F)--(A);
        \tikzstyle{f1}=[arrows={Arc Barb[reversed]-Arc Barb[reversed]}]
        \draw[darkgray][f1] ($(A) + (e1)$)--($(B) + (e1)$);
        \draw[darkgray][f1] ($(B) + (e2)$)--($(C) + (e2)$);
        \draw[darkgray][f1] ($(C) + (e3)$)--($(D) + (e3)$);
        \draw[darkgray][f1] ($(D) + (e4)$)--($(E) + (e4)$);
        \draw[darkgray][f1] ($(E) + (e5)$)--($(F) + (e5)$);
        \draw[darkgray][f1] ($(F) + (e6)$)--($(A) + (e6)$);
        \fill[darkgray] (0:1+\gap) circle (0.05cm);
        \fill[darkgray] (60:1+\gap) circle (0.05cm);
        \fill[darkgray] (120:1+\gap) circle (0.05cm);
        \fill[darkgray] (180:1+\gap) circle (0.05cm);
        \fill[darkgray] (240:1+\gap) circle (0.05cm);
        \fill[darkgray] (300:1+\gap) circle (0.05cm);
        \node at (0, 1.5) {};
        \node at (0, -1.5) {};
      \end{tikzpicture}
      \captionof{figure}{}
      \label{decomposition}
    \end{minipage}
  \end{figure}

  Let $v$ be the shortest vector of the form $a-b$, where $a\in A$ and $b\in B$. Put differently, $v$ is the projection of the origin onto $A-B$. Set $T:=A\cap(B+v)$. Clearly, $T$ is precisely the set of all points in $A$ nearest to $B$. Therefore $S\subset T$. Moreover, $T$ is a convex polyhedron inside the supporting hyperplane of $A$ that is perpendicular to $v$ (see Figure~\ref{fig_t}.)

  We decompose the polyhedron $T$ into the relative interiors of its faces (see Figure~\ref{decomposition}.) Let $e$ be the largest integer such that the relative interior of some $e$-dimensional face $F_{e}$ intersects $S$, say at point $s$. We shall prove that $(a_{2k})$ converges to $s$. Since, by Lemma~\ref{uniform}, $\lim_{k \to \infty} \abs{b_{2k+1} - P_{B}(a_{2k})} = 0$, this will imply that $(b_{2k+1})$ converges to $P_{B}(s)$, and $(a_{2k},b_{2k+1})$ thus converges to the best approximation pair $(s,P_{B}(s))$.

  The proof that $\lim_{k\to\infty} a_{2k} = s$ combines ideas from the two extreme cases for $e$, namely $e=0$ or $e=d-1$. We first handle these two cases and then present the general case.

  \paragraph{Case 1: $e=0$.} Suppose that all points of $S$ are vertices of $T$. Let $\epsilon_{0}>0$ be such that $N_{\epsilon_{0}}(s)$ (the $\epsilon_{0}$-neighborhood of $s$) satisfies $N_{\epsilon_{0}}(s)\cap T_{0}=\sset{s}$. Denoting by $T_{0}$ the set of all accumulation points of $(a_{2k})$, $S\subset T_{0}$ implies that every neighborhood of $T_{0}$ contains all but finitely many points of $(a_{2k})$. For every $\epsilon\in(0,\epsilon_{0}/4)$, we can then choose $k_{0}\in\N$ so that
  \begin{equation*}
    \abs{a_{2k_{0}}-s} < \frac{\epsilon}{4}, \quad \inf\abs{a_{2k}-T_{0}} <
    \frac{\epsilon}{4} \text{ for all } k \ge k_{0},
  \end{equation*}
  and, using Lemma~\ref{uniform},
  \begin{equation}\label{unifa}
    \sup_{b,x\in C}\abs{Q_{A,n}(b;x)-P_{A}(b)} < \frac{\epsilon}{4}, \;
    \sup_{a,x\in C}\abs{Q_{B,n}(a;x)-P_{B}(a)} < \frac{\epsilon}{4},
    \text{ for all } n \ge n_{2k_{0}}.
  \end{equation}
  We claim that for every $k\ge k_{0}$, if $\abs{a_{2k}-s}<\epsilon/4$, then $\abs{a_{2k+2}-s}<\epsilon_{0}/2$ (see Figure~\ref{e0}) and so $\abs{a_{2k+2}-s}=\inf\abs{a_{2k+2}-T_{0}}$, by the choice of $\epsilon_{0}$. In fact, \eqref{triangle} implies
  \begin{equation*}
    \abs{a_{2k+2}-s} \le \abs{a_{2k}-s} + \abs{P_{B}(a_{2k}) - b_{2k+1}} +
    \abs{P_{A}(b_{2k+1})-a_{2k+2}} < \epsilon/4 + \epsilon/4 + \epsilon/4 <
    \epsilon_{0}/2.
  \end{equation*}
  Hence, $N_{\epsilon}(a_{2k+2})\subset N_{\epsilon_{0}}(s)$, and so
  \begin{equation*}
    N_{\epsilon}(a_{2k+2}) \cap T_{0} \subset N_{\epsilon_{0}}(s) \cap
    T_{0}=\sset{s}.
  \end{equation*}
  This means that
  \begin{equation*}
    \abs{a_{2k+2}-s} = \inf\abs{a_{2k+2}-T_{0}} < \epsilon/4.
  \end{equation*}
  By induction, we know that $\abs{a_{2k}-s} < \epsilon/4$ for all $k \ge k_{0}$.

  \begin{figure}
    \centering
    \begin{minipage}[c]{0.4\textwidth}
      \centering
      \begin{tikzpicture}
        \clip(-3,-1.5) rectangle (3,1.5);
        \coordinate (A) at (0,0);
        \coordinate (G) at (0.6, 0);
        \coordinate (H) at (1.3, 0);
        \fill[darkgray] (A) circle (0.05);
        \draw[darkgray, dashed] (A) circle (3);
        \draw[darkgray, dashed] (A)--node[above] {$\epsilon_0$} (-3,0);
        \draw[darkgray, dashed] (A) circle (0.7);
        \fill[darkgray] (G) circle (0.05);
        \fill[darkgray] (H) circle (0.05);
        \draw[darkgray, dashed] (H) circle (1.4);
        \node at (A) [below right] {$s$};
        \node at (G) [below right] {$a_{2k}$};
        \node at (H) [below right] {$a_{2k+2}$};
      \end{tikzpicture}
      \captionof{figure}{}
      \label{e0}
    \end{minipage}%
    \begin{minipage}{0.6\textwidth}
      \centering
      \begin{tikzpicture}
        \coordinate (A) at (-2, 0);
        \coordinate (B) at (2, 0);
        \coordinate (C) at (-1.5, -0.5);
        \coordinate (D) at (1.5, -1.2);
        \coordinate (E) at (1.6, -0.75);
        \node at (A) [left] {$s$};
        \node at (B) [right] {$P_B(s)$};
        \fill[litegray, opacity=0.3] (-2, 1.5)--(-2, -1.5)--(-4, -1.5)--(-4, 1.5);
        \fill[litegray, opacity=0.3] (2, 1.5)--(2, -1.5)--(4, -1.5)--(4, 1.5);
        \fill[darkgray] (A) circle (0.05);
        \fill[darkgray] (B) circle (0.05);
        \draw[darkgray, latex-] (A)--node[above] {$v$} (B);
        \node[above] at (-3,-1.5) {$A$};
        \node[above] at (3,-1.5) {$B$};
        \draw[darkgray] (-2,1.5)--(-2,-1.5);
        \draw[darkgray, dashed] (B) circle (1.4);
        \node[above, fill=white, rounded corners=2pt, inner sep=1pt] at (-2,1.15) {$T\subset H$};
        \draw[darkgray, dashed] (-2.4, 0)--(-4, 0) node[left] {$T^\perp$};
        \draw[darkgray, dashed] (3.2,0)--(4,0);
        \draw[darkgray, dashed] (4, -0.5)--(-4, -0.5) node[left] {$T^\perp(a_{2k})$};
        \fill[darkgray] (C) circle (0.05) node[below] {$a_{2k}$};
        \fill[darkgray] (D) circle (0.05);
        \fill[darkgray] (E) circle (0.05) node[right] {$b_{2k+1}$};
        \draw[darkgray, -latex, decorate, decoration={
          zigzag, segment length=4, amplitude=.9, post=lineto,
          post length=2pt}] (D)--(E);
        \node[darkgray, fill=white, rounded corners=2pt, inner sep=1pt] at (1.25,-1.2) {$q_l$};
        \node at (0,1.5) {};
        \node at (0,-1.5) {};
      \end{tikzpicture}
      \captionof{figure}{}
      \label{e1}
    \end{minipage}
  \end{figure}

  \paragraph{Case 2: $e=d-1$.} Suppose $s$ is in the relative interior of the $(d-1)$-dimensional face $T=T_{e}$. Let $H$ be the supporting hyperplane of $A$ that is perpendicular to $v$. Note that $T\subset A\cap H$. If $\partial A_{i}$ goes through $s$, then $\partial A_{i}=H$ (otherwise, $T$ would have dimension smaller than $d-1.$) Similarly, if $\partial B_{i}$ goes through $P_{B}(s)=s-v$, then $\partial B_{i}=H-v$.

  Thus, we can choose an $\epsilon_{0}$-neighborhood $N_{\epsilon_{0}}(s)$ of $s$ such that $N_{\epsilon_{0}}(s)\subset A_{i}$ for all $\partial A_{i}\neq H$ and such that $N_{\epsilon_{0}}(P_{B}(s))\subset B_{i}$ for all $\partial B_{i}\neq H-v$. This, in particular, implies that for every $x\in N_{\epsilon_{0}}(P_{B}(s))$, $P_{B_{i}}(x)-x$ is always orthogonal to $T$.

  For the rest of the proof of this case for $e=d-1$, we only highlight the key steps and leave the full-fledged proof to the general case below. For a fixed $k\ge k_{0}$, let $T^{\perp}$ and $T^{\perp}(a_{2k+2})$ be the lines orthogonal to $T$ through $s$ and through $a_{2k}$, respectively, and define $\ell^{+} := \ell + n_{2k_{0}}$ and $q_{\ell} := Q_{B,\ell^{+}}(a_{2k};a'_{2k})$ (see Figure~\ref{e1}). We claim that if $q_{\ell}\in N_{\epsilon_{0}}(P_{B}(s))$ for all $\ell \ge 0$, then
  \begin{equation*}
    \inf\abs{b_{2k+1}-T^{\perp}(a_{2k})} =
    \inf\abs{q_{0}-T^{\perp}(a_{2k})} \cdot
    \prod_{n>n_{2k_{0}}}^{n_{2k}}(1-\lambda_{n}).
  \end{equation*}
  Indeed, recall that for $\ell \ge 1$, we have the recursion
  \begin{equation}\label{recursion}
    q_{\ell} = \lambda_{\ell^{+}}a_{2k} +
    (1-\lambda_{\ell^{+}})P_{B_{\ell^{+}}}(q_{\ell-1}).
  \end{equation}
  Since $q_{\ell-1}\in N_{\epsilon_{0}}(P_{B}(s))$, by the choice of $\epsilon_{0}$, we know that $P_{B_{\ell^{+}}}(q_{\ell-1})-q_{\ell-1}$ is orthogonal to $T$, hence
  \begin{equation*}
    \inf\abs{P_{B_{\ell^{+}}}(q_{\ell-1})-T^{\perp}} =
    \inf\abs{q_{\ell-1}-T^{\perp}}.
  \end{equation*}
  Therefore,
  \begin{equation*}
    \inf\abs{q_{\ell}-T^{\perp}} =
    (1-\lambda_{\ell^{+}})\inf\abs{q_{\ell-1}-T^{\perp}}.
  \end{equation*}
  The claim follows from repeated application of this equality and the fact that $b_{2k+1}=q_{n_{2k}-n_{2k_{0}}}$.

  Loosely speaking, the claim suggests that $b_{2k+1}$ does not deviate far from $T^{\perp}(a_{2k})$, and, similarly, neither does $a_{2k+2}$ deviate from $T^{\perp}(b_{2k+1})$. The accumulated deviations of $a_{2k}$ from $T^{\perp}$ then can be bounded. This, together with the fact that $a_{2k}$ converges to a point in $T$, implies that $\lim_{k\to\infty}a_{2k}=s$.

  \paragraph{General case: $0\le e\le d-1$.} The proof of the general case is the juxtaposition of the ideas in the two previos cases above. Let $T_{e}$ be the union of the $e$-dimensional faces of $T$. Since $s$ is in the relative interior of $F_{e}$, we can choose $\epsilon_{0}>0$ so that the following hold:
  \begin{enumerate}[noitemsep]
    \item $N_{\epsilon_{0}}(s)\cap T_{e}\subset F_{e}$;
    \item for every $x\in N_{\epsilon_{0}}(s)$, $P_{A_{m}}(x)-x$ is orthogonal
    to $F_{e}$ for all $m$;
    \item for every $x\in N_{\epsilon_{0}}(P_{B}(s))$, $P_{B_{n}}(x)-x$ is
    orthogonal to $F_{e}$ for all $n$.
  \end{enumerate}
  In view of \eqref{iter}, there is a constant $Z > 1$ such that
  \begin{equation}\label{bound}
    \sum_{l>k_{0}} \prod_{n>n_{2k_{0}}}^{n_{2l+1}}(1-\lambda_{n}) < Z
    \text{ for all } k_{0} \in \N.
  \end{equation}
  For every $\epsilon\in(0,\epsilon_{0}/4)$, we can choose $k_{0}\in\N$ so that
  \eqref{unifa} and the following hold.
  \begin{equation*}
    \abs{a_{2k_{0}}-s} < \frac{\epsilon}{4Z}, \quad \inf\abs{a_{2k}-T_{e}} <
    \frac{\epsilon}{4Z} \text{ for all } k \ge k_{0}.
  \end{equation*}
  Let $F_{e}^{\perp}$ be the $(d-e)$-dimensional affine subspace through $s$ orthogonal to $F_{e}$. We prove by induction that for all $k\ge k_{0}$,
  \begin{subequations}\label{induction}
    \begin{align}
      \inf\abs{a_{2k}-F_{e}} & < \frac{\epsilon}{4Z},\label{inda}\\
      \inf\abs{a_{2k}-F_{e}^{\perp}} & < \frac{\epsilon}{4Z} +
      \frac{\epsilon}{2Z} \cdot
      \sum_{l>k_{0}}^{k}\prod_{n>n_{2k_{0}}}^{n_{2l+1}}(1-\lambda_{m}),
      \label{indb}\\
      \abs{a_{2k}-s} & < \epsilon.\label{indc}
    \end{align}
  \end{subequations}
  When $k=k_{0}$, it is obvious from the choice of $k_{0}$. Assume that \eqref{induction} holds for $k$, we prove it for $k+1$.

  \noindent \textbf{Proof of \eqref{inda}.} We use \eqref{triangle} to obtain
  \begin{equation*}
    \abs{s-a_{2k+2}} \le \abs{s-a_{2k}} + \abs{P_{B}(a_{2k})-b_{2k+1}} +
    \abs{P_{A}(b_{2k+1})-a_{2k+2}} < \epsilon + \frac{\epsilon}{4Z} + \frac{\epsilon}{4Z} <
    \frac{\epsilon_0}{2}.
  \end{equation*}
  Hence, $N_{2\epsilon}(a_{2k+2})\subset N_{\epsilon_{0}}(s)$, and so
  \begin{equation*}
    N_{2\epsilon}(a_{2k+2})\cap T_{e} \subset N_{\epsilon_{0}}(s)\cap
    T_{e}\subset F_{e}.
  \end{equation*}
  This means that
  \begin{equation*}
    \inf\abs{a_{2k+2}-F_{e}}=\inf\abs{a_{2k+2}-T_{e}}<\frac{\epsilon}{4Z}.
  \end{equation*}

  \begin{figure}
    \centering \begin{tikzpicture}
      \coordinate (A) at (-0,-3);
      \coordinate (O) at (0,0);
      \coordinate (B) at (0,1);
      \coordinate (S) at (0,-1.4);
      \coordinate (T) at (1,-1.6);
      \coordinate (D) at (4.5,-2.5);
      \coordinate (E) at (4.6,-2);
      \fill[litegray, opacity=0.3] (O)++(-1, -0.5) -- ++(-1,1) -- ++(3,0) -- ++(1,-1);
      \node at (-1.5, 0) [left] {$F_e^\perp$};
      \fill[litegray, opacity=0.3] (S)++(-1, -0.5) -- ++(-1,1) -- ++(8,0) -- ++(1,-1);
      \node at (-1.5,-1.4) [left] {$F_e^\perp(a_{2k})$};
      \fill[darkgray] (O) circle (0.05) node[left] {$s$};
      \fill[darkgray] (S) circle (0.05) node[left] {$s'$};
      \draw[darkgray] (A) -- (B) node[left] {$F_e$};
      \fill[darkgray] (T) circle (0.05) node[right] {$a_{2k}$};
      \draw[darkgray, dashed] (S) -- (T);
      \fill[darkgray] (S)++(5,0) circle (0.05) node[right] {$P_B(s')$};
      \draw[darkgray, -latex, decorate, decoration={
        zigzag, segment length=4, amplitude=.9, post=lineto,
        post length=2pt}] (D) -- (E);
      \fill[darkgray] (D) circle (0.05) node[left] {$q_{l}$};
      \fill[darkgray] (E) circle (0.05) node[right] {$b_{2k+1}$};
    \end{tikzpicture}
    \captionof{figure}{}
    \label{general}
  \end{figure}

  \noindent \textbf{Proof of \eqref{indb}.} Again, define
  $\ell^{+}:=\ell+n_{2k_{0}}$ and $q_{\ell}:=Q_{B,\ell^{+}}(a_{2k};a'_{2k})$.
  For every $\ell\ge0$, we have
  \begin{equation}\label{ql}
    \abs{P_{B}(s)-q_{\ell}} \le \abs{s-a_{2k}} + \abs{P_{B}(a_{2k})-q_{\ell}} <
    \epsilon + \frac{\epsilon}{4Z} < \epsilon_{0}.
  \end{equation}
  Let $F_{e}^{\perp}(a_{2k})$ be the $(d-k)$-dimensional subspace through $a_{2k}$ orthogonal to $F_{e}$ (see Figure~\ref{general}). We have the recursion \eqref{recursion} for $\ell \ge 1$. Since $\abs{P_{B}(s)-q_{\ell-1}} < \epsilon_{0}$, by the choice of $\epsilon_{0}$, we know that $P_{B_{\ell^{+}}}(q_{\ell-1})-q_{\ell-1}$ is orthogonal to $F_{e}$, hence,
  \begin{equation*}
    \inf\abs{P_{B_{\ell^{+}}}(q_{\ell-1})-F_{e}^{\perp}(a_{2k})} =
    \inf\abs{q_{\ell-1}-F_{e}^{\perp}(a_{2k})}.
  \end{equation*}
  Therefore,
  \begin{equation}\label{step}
    \inf\abs{q_{\ell}-F_{e}^{\perp}(a_{2k})} =
    (1-\lambda_{\ell^{+}})\inf\abs{q_{\ell-1}-F_{e}^{\perp}(a_{2k})}.
  \end{equation}
  Let $s'$ be the intersection of $F_{e}$ and $F_{e}^{\perp}(a_{2k})$, which is guaranteed to be nonempty by $\abs{a_{2k}-s}<\epsilon<\epsilon_{0}$. Moreover, we have
  \begin{equation*}
    \abs{a_{2k}-s'}=\inf\abs{a_{2k}-F_{e}}<\frac{\epsilon}{4Z}.
  \end{equation*}
  Since $s'\in T$, $P_{B}(s')=s'+v\in F_{e}^{\perp}(a_{2k})$. Notice that
  \begin{equation*}
    \inf\abs{q_{0}-F_{e}^{\perp}(a_{2k})} \le \abs{q_{0}-P_{B}(s')} \le
    \abs{q_{0}-P_{B}(a_{2k})} + \abs{a_{2k}-s'} \le \frac{\epsilon}{4Z} + \frac{\epsilon}{4Z} =
    \frac{\epsilon}{2Z}.
  \end{equation*}
  With repeated application of \eqref{step} and the fact that $b_{2k+1} = q_{n_{2k}-n_{2k_{0}}}$, we derive
  \begin{equation}\label{dist-b}
    \inf\abs{b_{2k+1}-F_{e}^{\perp}(a_{2k})} \le \frac{\epsilon}{2Z} \cdot
    \prod_{n>n_{2k_{0}}}^{n_{2k}}(1-\lambda_{m}).
  \end{equation}
  Let $F_{e}^{\perp}(b_{2k+1})$ be the $(d-k)$-dimensional subspace through $b_{2k+1}$ orthogonal to $F_{e}$. Notice that \eqref{ql} implies that
  \begin{equation*}
    \abs{b_{2k+1}-P_{B}(s)}<\epsilon+\frac{\epsilon}{4Z}.
  \end{equation*}
  This would allow us to carry out a similar argument to conclude that
  \begin{equation}\label{dist-a}
    \inf\abs{a_{2k+2}-F_{e}^{\perp}(b_{2k+1})} <
    \frac{\epsilon}{2Z}\cdot\prod_{n>n_{2k_{0}}}^{n_{2k+1}}(1-\lambda_{m}).
  \end{equation}
  Combining \eqref{dist-b} and \eqref{dist-a}, we get what is needed for the inductive step.

  \noindent \textbf{Proof of \eqref{indc}.} It follows from
  \begin{equation*}
    \abs{a_{2k+2}-s} \le \inf\abs{a_{2k+2}-F_{e}} +
    \inf\abs{a_{2k+2}-F_{e}^{\perp}}
  \end{equation*}
  and (\ref{bound}, \ref{inda}, \ref{indb}).
\end{proof}

\begin{remark}\label{cor}
  As mentioned in the proof, if the polyhedra $A$ and $B$ are known a priori to have only one best approximation pair, then the set $S$ is automatically a singleton, hence $(a_{2k}, b_{2k+1})$ converges to the best approximation pair. In this case, the second inequality in \eqref{iter} could be dropped from the assumptions in Theorem~\ref{main}.
\end{remark}

\section{Discussion}\label{discussion}

\begin{figure}
  \centering
  \begin{minipage}{0.5\textwidth}
    \centering
    \begin{tikzpicture}[scale=0.3]
      \clip(-12,-14) rectangle (12,14);
      \draw[litegray, dashed] (-12,-14) -- (12,-14) -- (12,14) -- (-12,14) -- cycle;
      \coordinate (A) at (-8,-9);
      \coordinate (A0) at (-7,-15);
      \coordinate (A1) at (-4,-11);
      \coordinate (A2) at (-4,-7);
      \coordinate (A3) at (-6,-5);
      \coordinate (A4) at (-7,-5);
      \fill[red, opacity=0.2] (A0)++(-5,0) -- (A0) -- (A1) -- (A2) -- (A3) -- (A4) -- ++(-5,0);
      \fill[darkgray] (8, -13) circle (0.17) node[right] {$a_0$};
      \node at (A) {$A$};
      \coordinate (B0) at (18,15);
      \coordinate (B1) at (10,5);
      \coordinate (B2) at (8,4);
      \coordinate (B3) at (4,5);
      \coordinate (B4) at (-1,15);
      \coordinate (B) at (7,9);
      \fill[blue, opacity=0.2] (B0) -- (B0) -- (B1) -- (B2) -- (B3) -- (B4);
      \node at (B) {$B$};
      \draw[red, opacity=0.0] plot [only marks, mark size=4, mark=*] coordinates {(8.0, -13.0)};
      \draw[red, opacity=0.04] plot [only marks, mark size=4, mark=*] coordinates {(0.3, -9.6)};
      \draw[red, opacity=0.08] plot [only marks, mark size=4, mark=*] coordinates {(-1.625, -8.75)};
      \draw[red, opacity=0.12] plot [only marks, mark size=4, mark=*] coordinates {(-2.1063, -8.5375)};
      \draw[red, opacity=0.16] plot [only marks, mark size=4, mark=*] coordinates {(-2.2266, -8.4844)};
      \draw[red, opacity=0.2] plot [only marks, mark size=4, mark=*] coordinates {(-2.0628, -4.818)};
      \draw[red, opacity=0.24] plot [only marks, mark size=4, mark=*] coordinates {(-2.3671, -4.1471)};
      \draw[red, opacity=0.28] plot [only marks, mark size=4, mark=*] coordinates {(-0.2942, -3.1659)};
      \draw[red, opacity=0.32] plot [only marks, mark size=4, mark=*] coordinates {(0.7432, -2.0749)};
      \draw[red, opacity=0.36] plot [only marks, mark size=4, mark=*] coordinates {(-2.8412, -1.6707)};
      \draw[red, opacity=0.4] plot [only marks, mark size=4, mark=*] coordinates {(-5.3726, -3.7201)};
      \draw[red, opacity=0.44] plot [only marks, mark size=4, mark=*] coordinates {(-5.4175, -3.8094)};
      \draw[red, opacity=0.48] plot [only marks, mark size=4, mark=*] coordinates {(-4.0984, -2.9737)};
      \draw[red, opacity=0.52] plot [only marks, mark size=4, mark=*] coordinates {(-5.3648, -4.2201)};
      \draw[red, opacity=0.56] plot [only marks, mark size=4, mark=*] coordinates {(-3.6528, -2.8058)};
      \draw[red, opacity=0.6] plot [only marks, mark size=4, mark=*] coordinates {(-4.9869, -4.3838)};
      \draw[red, opacity=0.64] plot [only marks, mark size=4, mark=*] coordinates {(-4.9864, -4.5454)};
      \draw[red, opacity=0.68] plot [only marks, mark size=4, mark=*] coordinates {(-5.066, -4.6148)};
      \draw[red, opacity=0.72] plot [only marks, mark size=4, mark=*] coordinates {(-4.8811, -4.3541)};
      \draw[red, opacity=0.76] plot [only marks, mark size=4, mark=*] coordinates {(-5.5908, -4.8733)};
      \draw[red, opacity=0.8] plot [only marks, mark size=4, mark=*] coordinates {(-4.9605, -4.3153)};
      \draw[red, opacity=0.84] plot [only marks, mark size=4, mark=*] coordinates {(-5.0963, -4.4198)};
      \draw[red, opacity=0.88] plot [only marks, mark size=4, mark=*] coordinates {(-5.4201, -4.6735)};
      \draw[red, opacity=0.92] plot [only marks, mark size=4, mark=*] coordinates {(-5.5062, -4.727)};
      \draw[red, opacity=0.96] plot [only marks, mark size=4, mark=*] coordinates {(-5.5734, -4.7768)};
      \draw[red, opacity=1.0] plot [only marks, mark size=4, mark=*] coordinates {(-5.8257, -4.9922)};
      \draw[blue, opacity=0.02] plot [only marks, mark size=4, mark=square*] coordinates {(4.6, -6.2)};
      \draw[blue, opacity=0.06] plot [only marks, mark size=4, mark=square*] coordinates {(0.75, -4.5)};
      \draw[blue, opacity=0.1] plot [only marks, mark size=4, mark=square*] coordinates {(-0.2125, -4.075)};
      \draw[blue, opacity=0.14] plot [only marks, mark size=4, mark=square*] coordinates {(-0.4531, -3.9688)};
      \draw[blue, opacity=0.18] plot [only marks, mark size=4, mark=square*] coordinates {(0.4953, 0.8625)};
      \draw[blue, opacity=0.22] plot [only marks, mark size=4, mark=square*] coordinates {(0.3137, 2.1437)};
      \draw[blue, opacity=0.26] plot [only marks, mark size=4, mark=square*] coordinates {(2.5966, 2.3365)};
      \draw[blue, opacity=0.3] plot [only marks, mark size=4, mark=square*] coordinates {(3.2493, 2.3128)};
      \draw[blue, opacity=0.34] plot [only marks, mark size=4, mark=square*] coordinates {(2.953, 1.6587)};
      \draw[blue, opacity=0.38] plot [only marks, mark size=4, mark=square*] coordinates {(0.4726, 1.8782)};
      \draw[blue, opacity=0.42] plot [only marks, mark size=4, mark=square*] coordinates {(-0.4211, 4.5248)};
      \draw[blue, opacity=0.46] plot [only marks, mark size=4, mark=square*] coordinates {(1.3691, 4.1185)};
      \draw[blue, opacity=0.5] plot [only marks, mark size=4, mark=square*] coordinates {(1.1835, 3.3819)};
      \draw[blue, opacity=0.54] plot [only marks, mark size=4, mark=square*] coordinates {(2.7157, 5.2395)};
      \draw[blue, opacity=0.58] plot [only marks, mark size=4, mark=square*] coordinates {(1.9745, 4.8585)};
      \draw[blue, opacity=0.62] plot [only marks, mark size=4, mark=square*] coordinates {(2.1738, 4.0929)};
      \draw[blue, opacity=0.66] plot [only marks, mark size=4, mark=square*] coordinates {(2.5017, 4.2441)};
      \draw[blue, opacity=0.7] plot [only marks, mark size=4, mark=square*] coordinates {(2.7161, 4.3662)};
      \draw[blue, opacity=0.74] plot [only marks, mark size=4, mark=square*] coordinates {(2.8366, 4.9175)};
      \draw[blue, opacity=0.78] plot [only marks, mark size=4, mark=square*] coordinates {(3.0603, 4.5856)};
      \draw[blue, opacity=0.82] plot [only marks, mark size=4, mark=square*] coordinates {(3.2767, 4.6702)};
      \draw[blue, opacity=0.86] plot [only marks, mark size=4, mark=square*] coordinates {(3.321, 4.9571)};
      \draw[blue, opacity=0.9] plot [only marks, mark size=4, mark=square*] coordinates {(3.4201, 4.9654)};
      \draw[blue, opacity=0.94] plot [only marks, mark size=4, mark=square*] coordinates {(3.6796, 4.9313)};
      \draw[blue, opacity=0.98] plot [only marks, mark size=4, mark=square*] coordinates {(3.6354, 4.8438)};
    \end{tikzpicture}
    \captionof{figure}{}
    \label{exp1}
  \end{minipage}%
  \begin{minipage}[c]{0.5\textwidth}
    \centering
    \begin{tikzpicture}[scale=0.3]
      \clip(-12,-14) rectangle (12,14);
      \draw[litegray, dashed] (-12,-14) -- (12,-14) -- (12,14) -- (-12,14) -- cycle;
      \coordinate (A) at (-8,-9);
      \coordinate (A0) at (-7,-15);
      \coordinate (A1) at (-4,-11);
      \coordinate (A2) at (-4,-7);
      \coordinate (A3) at (-6,-5);
      \coordinate (A4) at (-7,-5);
      \fill[red, opacity=0.2] (A0)++(-5,0) -- (A0) -- (A1) -- (A2) -- (A3) -- (A4) -- ++(-5,0);
      \fill[darkgray] (8, -13) circle (0.17) node[right] {$a_0$};
      \node at (A) {$A$};
      \coordinate (B0) at (18,15);
      \coordinate (B1) at (10,5);
      \coordinate (B2) at (8,4);
      \coordinate (B3) at (4,5);
      \coordinate (B4) at (-1,15);
      \coordinate (B) at (7,9);
      \fill[blue, opacity=0.2] (B0) -- (B0) -- (B1) -- (B2) -- (B3) -- (B4);
      \node at (B) {$B$};
      \draw[red, opacity=0.0] plot [only marks, mark size=4, mark=*] coordinates {(8.0, -13.0)};
      \draw[red, opacity=0.04] plot [only marks, mark size=4, mark=*] coordinates {(0.3, -9.6)};
      \draw[red, opacity=0.08] plot [only marks, mark size=4, mark=*] coordinates {(-1.625, -7.05)};
      \draw[red, opacity=0.12] plot [only marks, mark size=4, mark=*] coordinates {(-2.7063, -5.4375)};
      \draw[red, opacity=0.16] plot [only marks, mark size=4, mark=*] coordinates {(-3.3416, -4.4106)};
      \draw[red, opacity=0.2] plot [only marks, mark size=4, mark=*] coordinates {(-4.9679, -2.074)};
      \draw[red, opacity=0.24] plot [only marks, mark size=4, mark=*] coordinates {(-6.2018, -0.8914)};
      \draw[red, opacity=0.28] plot [only marks, mark size=4, mark=*] coordinates {(-3.8235, -2.3683)};
      \draw[red, opacity=0.32] plot [only marks, mark size=4, mark=*] coordinates {(-1.5509, -1.4677)};
      \draw[red, opacity=0.36] plot [only marks, mark size=4, mark=*] coordinates {(-3.0433, -1.2538)};
      \draw[red, opacity=0.4] plot [only marks, mark size=4, mark=*] coordinates {(-5.4684, -3.6025)};
      \draw[red, opacity=0.44] plot [only marks, mark size=4, mark=*] coordinates {(-5.4764, -3.8028)};
      \draw[red, opacity=0.48] plot [only marks, mark size=4, mark=*] coordinates {(-4.1367, -2.9531)};
      \draw[red, opacity=0.52] plot [only marks, mark size=4, mark=*] coordinates {(-5.3807, -4.2033)};
      \draw[red, opacity=0.56] plot [only marks, mark size=4, mark=*] coordinates {(-3.6569, -2.8027)};
      \draw[red, opacity=0.6] plot [only marks, mark size=4, mark=*] coordinates {(-4.987, -4.3838)};
      \draw[red, opacity=0.64] plot [only marks, mark size=4, mark=*] coordinates {(-4.9866, -4.5453)};
      \draw[red, opacity=0.68] plot [only marks, mark size=4, mark=*] coordinates {(-5.066, -4.6148)};
      \draw[red, opacity=0.72] plot [only marks, mark size=4, mark=*] coordinates {(-4.8811, -4.3541)};
      \draw[red, opacity=0.76] plot [only marks, mark size=4, mark=*] coordinates {(-5.5908, -4.8733)};
      \draw[red, opacity=0.8] plot [only marks, mark size=4, mark=*] coordinates {(-4.9605, -4.3153)};
      \draw[red, opacity=0.84] plot [only marks, mark size=4, mark=*] coordinates {(-5.0963, -4.4198)};
      \draw[red, opacity=0.88] plot [only marks, mark size=4, mark=*] coordinates {(-5.4201, -4.6735)};
      \draw[red, opacity=0.92] plot [only marks, mark size=4, mark=*] coordinates {(-5.5062, -4.727)};
      \draw[red, opacity=0.96] plot [only marks, mark size=4, mark=*] coordinates {(-5.5734, -4.7768)};
      \draw[red, opacity=1.0] plot [only marks, mark size=4, mark=*] coordinates {(-5.8257, -4.9922)};
      \draw[blue, opacity=0.02] plot [only marks, mark size=4, mark=square*] coordinates {(4.6, -6.2)};
      \draw[blue, opacity=0.06] plot [only marks, mark size=4, mark=square*] coordinates {(0.75, -4.5)};
      \draw[blue, opacity=0.1] plot [only marks, mark size=4, mark=square*] coordinates {(-1.4125, -3.825)};
      \draw[blue, opacity=0.14] plot [only marks, mark size=4, mark=square*] coordinates {(-2.6831, -3.3838)};
      \draw[blue, opacity=0.18] plot [only marks, mark size=4, mark=square*] coordinates {(-1.8359, 2.7103)};
      \draw[blue, opacity=0.22] plot [only marks, mark size=4, mark=square*] coordinates {(-2.8992, 3.6195)};
      \draw[blue, opacity=0.26] plot [only marks, mark size=4, mark=square*] coordinates {(0.4498, 5.5267)};
      \draw[blue, opacity=0.3] plot [only marks, mark size=4, mark=square*] coordinates {(1.7077, 3.8308)};
      \draw[blue, opacity=0.34] plot [only marks, mark size=4, mark=square*] coordinates {(1.7399, 2.4924)};
      \draw[blue, opacity=0.38] plot [only marks, mark size=4, mark=square*] coordinates {(0.3257, 2.1783)};
      \draw[blue, opacity=0.42] plot [only marks, mark size=4, mark=square*] coordinates {(-0.5781, 4.5774)};
      \draw[blue, opacity=0.46] plot [only marks, mark size=4, mark=square*] coordinates {(1.3104, 4.2112)};
      \draw[blue, opacity=0.5] plot [only marks, mark size=4, mark=square*] coordinates {(1.156, 3.42)};
      \draw[blue, opacity=0.54] plot [only marks, mark size=4, mark=square*] coordinates {(2.7078, 5.2541)};
      \draw[blue, opacity=0.58] plot [only marks, mark size=4, mark=square*] coordinates {(1.974, 4.8587)};
      \draw[blue, opacity=0.62] plot [only marks, mark size=4, mark=square*] coordinates {(2.1734, 4.0936)};
      \draw[blue, opacity=0.66] plot [only marks, mark size=4, mark=square*] coordinates {(2.5017, 4.2441)};
      \draw[blue, opacity=0.7] plot [only marks, mark size=4, mark=square*] coordinates {(2.7162, 4.3661)};
      \draw[blue, opacity=0.74] plot [only marks, mark size=4, mark=square*] coordinates {(2.8366, 4.9175)};
      \draw[blue, opacity=0.78] plot [only marks, mark size=4, mark=square*] coordinates {(3.0603, 4.5855)};
      \draw[blue, opacity=0.82] plot [only marks, mark size=4, mark=square*] coordinates {(3.2767, 4.6702)};
      \draw[blue, opacity=0.86] plot [only marks, mark size=4, mark=square*] coordinates {(3.321, 4.9571)};
      \draw[blue, opacity=0.9] plot [only marks, mark size=4, mark=square*] coordinates {(3.4201, 4.9654)};
      \draw[blue, opacity=0.94] plot [only marks, mark size=4, mark=square*] coordinates {(3.6796, 4.9313)};
      \draw[blue, opacity=0.98] plot [only marks, mark size=4, mark=square*] coordinates {(3.6354, 4.8438)};
    \end{tikzpicture}
    \captionof{figure}{}
    \label{exp2}
  \end{minipage}
\end{figure}

One has the freedom to choose the auxiliary sequence $(a_{2k}',b_{2k+1}')$ in the A-HLWB algorithm as long as it is bounded. The simplest way is to take $a_{2k}'=b_{2k+1}'=a_{0}$. In Figure \ref{exp1}, we run 50 iterations of the algorithm, and we plot the more recent points in darker color. The half-spaces are
\begin{align*}
  A_{1}\colon & 4x-3y \le 17,&
  A_{2}\colon & x \le -4, &
  A_{3}\colon & x + y \le -11, &
  A_{4}\colon & y \le -5; \\
  B_{1}\colon & 5x-4y \le 30, &
  B_{2}\colon & x-2y \le 0, &
  B_{3}\colon & -x - 4y \le -24, &
  B_{4}\colon & -2x - y \le -13.
\end{align*}
The choices of the parameters are
\begin{equation*}
  a_{0} = (8,-13), \quad \lambda_{n} = \tfrac{1}{n+1}, \quad
  n_{k} = \lfloor1.1^{k}\rfloor, \quad a_{2k}' = b_{2k+1}' = a_{0}.
\end{equation*}

The auxiliary point $a_{2k}'$ can be seen as the starting point of the HLWB algorithm for polyhedron $B$. It makes sense to choose $a_{2k}'$ close to $B$. Since $b_{2k-1}$ is our best approximation to $B$ so far, heuristically $a_{2k}'=b_{2k-1}$ might be a better choice. Similarly, it might be better to choose $b_{2k+1}'=a_{2k}$. One can use Lemma~\ref{compact} to show that $(a_{2k}',b_{2k+1}')$ is bounded. In Figure~\ref{exp2}, we again run 50 iterations for the same half-spaces and the parameters, except that
\begin{equation*}
  a_{0}' = a_{0}, \quad a_{2k}' = b_{2k-1} \text{ for all } k > 0, \quad
  b_{2k+1}' = a_{2k} \text{ for all } k \ge 0.
\end{equation*}

Comparing Figure~\ref{exp1} and Figure~\ref{exp2}, interestingly we do not see significant difference in convergence. One possible explanation is the following. It is easily checked that
\begin{equation*}
  \abs{Q_{B,n_{2k}}(a;x)-Q_{B,n_{2k}}(a;x')}\le\abs{x-x'}{\textstyle
  \prod_{n=1}^{n_{2k}}(1-\lambda_{n}).}
\end{equation*}
When $\prod_{n=1}^{n_{2k}}(1-\lambda_{n})$ is extremely small, the contribution of the auxiliary points is negligible.

Now that the convergence of the A-HLWB algorithm has been established here, it would be interesting in future work to investigate non-asymptotic bounds on the number of steps of half-space projections to reach an approximate solution and rate of convergence results.

\section*{Acknowledgements}

We thank Yehuda Zur for Matlab programming work at the early stages of our research.

\bibliographystyle{jalpha}
\bibliography{best_approx_pair}

\end{document}